\documentclass[12pt,a4paper,reqno]{amsart}

\usepackage[utf8]{inputenc}
\usepackage[english]{babel}

\usepackage{amsmath,amsthm}
\usepackage{amssymb}
\usepackage[style=numeric,block=ragged,backend=bibtex,sorting=nyt]{biblatex}
\usepackage{cases}
\theoremstyle{theorem}
\newtheorem{theorem}{Theorem}

\theoremstyle{definition}

\theoremstyle{remark}

\usepackage{url}
\usepackage[colorlinks,linkcolor=blue,citecolor=blue,breaklinks]{hyperref}
\hypersetup{breaklinks=true}

\title{Binary quadratic forms and sums of powers of integers}

\author[J.L. Cereceda]{Jos\'e Luis Cereceda}
\address{%
        Collado Villalba, 28400 -- Madrid, Spain}
\email{jl.cereceda@movistar.es}

\begin{document}

\begin{abstract}
In this methodological paper, we first review the classic cubic Diophantine equation $a^3 + b^3 + c^3 = d^3$, and consider
the specific class of solutions $q_1^3 + q_2^3 + q_3^3 = q_4^3$ with each $q_i$ being a binary quadratic form. Next we turn
our attention to the familiar sums of powers of the first $n$ positive integers, $S_k = 1^k + 2^k + \cdots + n^k$, and express the squares $S_k^2$, $S_m^2$, and the product $S_k S_m$ as a linear combination of power sums. These expressions, along with the
above quadratic-form solution for the cubic equation, allows one to generate an infinite number of relations of the form $Q_1^3 + Q_2^3 + Q_3^3 = Q_4^3$, with each $Q_i$ being a linear combination of power sums. Also, we briefly consider the quadratic
Diophantine equations $a^2 + b^2 + c^2 = d^2$  and $a^2 + b^2 = c^2$, and give a family of corresponding solutions $Q_1^2
+ Q_2^2 + Q_3^2 = Q_4^2$ and $Q_1^2 + Q_2^2 = Q_3^2$ in terms of sums of powers of integers.
\end{abstract}

\maketitle

\section{Introduction}\label{sec:1}

Our starting point is the cubic Diophantine equation
\begin{equation}\label{cubic}
a^3 + b^3 + c^3 = d^3,    \qquad  abcd \neq 0.
\end{equation}
(Note that $abcd \neq 0$ as, by Fermat's Last Theorem, we cannot have $a^3 + b^3 = c^3$.) As pointed out by Dickson in his
comprehensive {\it History of the Theory of Numbers}, the problem of finding the rational or integer (positive or negative) solutions to Equation \eqref{cubic} can be traced back to Diophantus \cite[p. 550]{dickson}. A first parametric solution was given by Vieta in 1591 \cite[p. 551]{dickson} and, in 1754, Euler found the most general family of rational solutions to \eqref{cubic} (see \cite[p. 552]{dickson} and \cite{cook}). Much more recently, Choudhry \cite{choudhry} obtained a {\it complete\/} solution of \eqref{cubic} in positive integers.

It is to be noted that there are several different formulations equivalent to the general solution discovered by Euler. In his third notebook, Ramanujan provided a family of solutions equivalent to Euler's general solution that appears to be the simplest of
all \cite{berndt,chamber}. In addition to this general solution, Ramanujan also gave some further families of parametric solutions
to \eqref{cubic} as well as several numerical examples. Specifically, in a problem submitted to the {\it Journal of the Indian
Mathematical Society}, (Question 441, JIMS 5, \mbox{p. 39}, 1913), Ramanujan put forward the following two-parameter
solution to Equation \eqref{cubic} \cite{berndt2}:
\begin{equation}\label{Rsol}
(3u^2 +5uv -5v^2)^3 + (4u^2 -4uv +6v^2)^3 + (5u^2 -5uv -3v^2)^3 = (6u^2 -4uv + 4v^2)^3.
\end{equation}
Relation \eqref{Rsol} constitutes an algebraic identity and, as such, is satisfied by any real or complex values of the parameters
$u$ and $v$. As we are dealing with Diophantine equations, however, it will be assumed that $u$ and $v$ take only rational or
integer values. In particular, putting $u=1$ and $v=0$ in \eqref{Rsol} gives us the smallest positive solution to \eqref{cubic},
namely $3^3 + 4^3 + 5^3 = 6^3$.

In this methodological paper, we search for solutions of the kind shown in Equation \eqref{Rsol}, that is, solutions $q_1^3 + q_2^3 + q_3^3 = q_4^3$ for which each of the $q_i$'s ($i=1,2,3,4$) adopts the form of a quadratic polynomial of two variables, say $u$ and $v$ (a binary quadratic form): $q_i = \alpha_i u^2 + \beta_i uv + \gamma_i v^2$, where $\alpha_i$, $\beta_i$, and $\gamma_i$ take integer (positive or negative) values. Our interest in this type of solutions stems from the fact that, as explained in \cite{harper}, by using the above identity $3^3 + 4^3 + 5^3 = 6^3$ as a seed, one can generate quadratic-form formulas to Equation \eqref{cubic}. Expanding on this point, and borrowing a theorem of S\'andor \cite[Theorem 1]{sandor}, in Section \ref{sec:2} we show that, indeed, it is possible to construct quadratic-form representations for the cubic equation \eqref{cubic} starting from {\it any\/} particular nontrivial solution to \eqref{cubic} (see below for the definition of a trivial solution). The proof of this result given by S\'andor (which is essentially reproduced in Section 2) is particularly suitable for our purpose since it utilizes only precalculus tools. Furthermore, S\'andor's theorem allows one to readily produce a wealth of algebraic identities like that in Equation \eqref{Rsol} by simply adding and multiplying the integers $a$, $b$, $c$, and $d$ constituting a particular (nontrivial) solution of \eqref{cubic}.

In Section \ref{sec:3}, we consider the familiar sums of powers of the first $n$ positive integers, $S_k = 1^k + 2^k + \cdots + n^k$ (with $k$ being a nonnegative integer), and express the squares $S_k^2$, $S_m^2$, and the product $S_k S_m$ as a linear combination of power sums. In this way, using such expressions for $S_k^2$, $S_m^2$, and $S_k S_m$, the following generic quadratic form
\begin{equation}\label{Qi}
Q_i(k,m,n) = \alpha_i S_k^2 + \beta_i S_k S_m + \gamma_i S_m^2,
\end{equation}
can be equally expressed as a linear combination of power sums. (Note that $Q_i(k,m,n)$ depends explicitly on $n$ through the
power sums $S_k$ and $S_m$.) Therefore, using the quadratic-form solutions obtained in Section 2, one can construct relationships
of the type $Q_{1}(k,m,n)^3 + Q_{2}(k,m,n)^3 + Q_{3}(k,m,n)^3 = Q_{4}(k,m,n)^3$, with each $Q_i(k,m,n)$ being a linear
combination of power sums. Moreover, substituting each of the power sums in $Q_i(k,m,n)$ for its polynomial representation yields
(for fixed $k$ and $m$) algebraic identities of the form $Q_1(u)^3 + Q_2(u)^3 + Q_3(u)^3 = Q_4(u)^3$, where each $Q_i(u)$ is
itself a polynomial in the real or complex variable $u$ (see, for instance, Equation \eqref{relation2} below).

Finally, in Section \ref{sec:4} we briefly consider the quadratic Diophantine equation $a^2 +b^2 + c^2 = d^2$. Using a particularly simple quadratic-form solution for this equation, we give a corresponding solution in terms of $S_k$ and $S_k^2$ (see Equation \eqref{quadratic3} below). On the other hand, starting from an almost trivial identity, we give a family of Pythagorean triangles whose side lengths are given by $|S_k^2 - S_m^2|$, $2S_k S_m$, and $S_k^2 +S_m^2$. As a by-product, we also obtain a family of solutions for the Diophantine equation $a^2 + b^2 = c^2 +d^2$.

From a pedagogical point of view, this methodological paper could be of interest to both high school and college students for the
following reasons. On the one hand, it shows in an elementary way how to obtain systematically quadratic-form solutions for the cubic equation \eqref{cubic}. In this regard, as we shall see, S\'andor's theorem proves to be extremely useful to this end since it provides a fairly elementary yet powerful method to generate quadratic-form formulas $q_1^3 + q_2^3 + q_3^3 = q_4^3$ for Equation \eqref{cubic}. On the other hand, we introduce some well-known formulas (though rarely found in the current literature) involving sums of powers of integers, in particular that expressing the product $S_k S_m$ as a linear combination of $S_j$'s. Using these formulas, and with the aid of a computer algebra system, students ought reliably compute the quadratic form in Equation \eqref{Qi} for a variety of values of the parameters. Last, but not least, equipped with the given formulas for $S_k S_m$, $S_k^2$, $S_1^k$, and $S_2 S_1^k$, students might want to explore other low degree Diophantine equations (see, in this respect, \cite[Chapter 2]{barbeau}) and recast some of their solutions in terms of sums of powers of integers.

\section{Quadratic solutions for the cubic equation}\label{sec:2}

As was anticipated in the introduction, we shall make use of a theorem of S\'andor (see \cite[Theorem 1]{sandor}) in order to construct two-parameter quadratic solutions for the cubic equation \eqref{cubic}. Following S\'andor, we say that a solution of \eqref{cubic} is trivial if $d=a$ or $d=b$ or $d =c$. The said theorem, adapted to our notation, is as follows.
\begin{theorem}\label{th:2.1}
If $(a,b,c,d)$ is a nontrivial integer solution of \eqref{cubic} then for any integer values of $u$ and $v$
\begin{equation}\label{sandor}
\begin{split}
q_1 & = a(a+c) u^2 + (d-b)(d+b) uv - c(d-b) v^2 ,  \\
q_2 & = b(a+c) u^2 - (c-a)(c+a) uv + d(d-b) v^2 ,  \\
q_3 & = c(a+c) u^2 - (d-b)(d+b) uv - a(d-b) v^2 ,   \\
q_4 & = d(a+c) u^2 - (c-a)(c+a) uv + b(d-b) v^2 ,
\end{split}
\end{equation}
satisfy $q_1^3 + q_2^3 + q_3^3 = q_4^3$.
\end{theorem}
\begin{proof}
As noted by S\' andor, relations \eqref{sandor} can be obtained by generalizing Ramanujan's quadratic solution \eqref{Rsol} but,
following \cite{sandor}, we will give a simpler proof of Theorem \ref{th:2.1} employing a technique devised by Nicholson \cite{nicholson}. Thus, let $a^3 + b^3 + c^3 = d^3$ be a nontrivial solution of \eqref{cubic}, and consider Nicholson's parametric equation \cite{sandor}
\begin{equation}\label{nichol}
\big( ux -cy \big)^3 + \big( -ux -ay \big)^3 + \big( vx -by \big)^3 = \big( vx -dy \big)^3 .
\end{equation}
Expanding in \eqref{nichol} and using the constraint $a^3 + b^3 + c^3 = d^3$, the cubic terms vanish and we are left with an equation involving only quadratic and linear exponents, namely
\begin{multline*}
-3u^2x^2cy + 3uxc^2y^2 -   3u^2x^2ay - 3uxa^2y^2 - 3v^2x^2by   \\
 + 3vxb^2y^2 = -3v^2x^2dy + 3vxd^2y^2.
\end{multline*}
Dividing throughout by the common factor $3xy$ gives
\begin{equation*}
x \big( dv^2 - bv^2 - cu^2 - au^2 \big) = y \big( d^2 v - b^2 v + a^2 u - c^2 u \big).
\end{equation*}
Clearly, the values $x = d^2 v - b^2 v + a^2 u - c^2 u$ and $y = dv^2 - bv^2 - cu^2 - au^2$ satisfy the equation, and then
we obtain
\begin{equation*}
\begin{split}
ux - cy  & = u \big( d^2 v - b^2 v + a^2 u - c^2 u \big) - c \big( dv^2 - bv^2 - cu^2 - au^2 \big) \\
& = a(a+c) u^2 + (d-b)(d+b) uv - c(d-b) v^2,  \\
-ux -ay & = -u \big( d^2 v - b^2 v + a^2 u - c^2 u \big) - a \big( dv^2 - bv^2 - cu^2 - au^2 \big) \\
& = b(a+c) u^2 - (c-a)(c+a) uv + d(d-b) v^2,  \\
vx -by & = v \big( d^2 v - b^2 v + a^2 u - c^2 u \big) - b \big( dv^2 - bv^2 - cu^2 - au^2 \big) \\
& = c(a+c) u^2 - (d-b)(d+b) uv - a(d-b) v^2,  \\
vx -dy & = v \big( d^2 v - b^2 v + a^2 u - c^2 u \big) - d \big( dv^2 - bv^2 - cu^2 - au^2 \big) \\
& =  d(a+c) u^2 - (c-a)(c+a) uv + b(d-b) v^2.
\end{split}
\end{equation*}
Nicholson's parametric equation \eqref{nichol} then guarantees that $q_1^3 + q_2^3 + q_3^3 = q_4^3$, with the $q_i$'s being
given by Equations \eqref{sandor}.
\end{proof}

Let us consider a few examples illustrating the application of Theorem \ref{th:2.1}. In each case, starting from a given (nontrivial) integer solution $(a,b,c,d)$ to \eqref{cubic}, it generates a two-parameter family of solutions $(q_1,q_2,q_3,q_4)$ satisfying \eqref{cubic}:
\begin{enumerate}

\item Substituting $(a,b,c,d) = (3,4,5,6)$ into Equations \eqref{sandor}, dividing by $2$, and using the linear transformation
$u \to u/2$, we get Ramanujan's solution \eqref{Rsol} \cite{sandor}.

\item Using $(a,b,c,d) = (1,6,8,9)$ into Equations \eqref{sandor} and dividing by $3$ yields
\begin{align}\label{ex2}
(3u^2 +15uv -8v^2)^3 & + (18u^2 -21uv +9v^2)^3  \notag  \\
& + (24u^2 -15uv -v^2)^3  = (27u^2 -21uv + 6v^2)^3.
\end{align}
Now, putting $u = 1$ and $v = 2$ in \eqref{ex2} gives $1^3 + 12^3 + (-10)^3 = 9^3$ or, equivalently,
\begin{equation}\label{ex21}
1^3 + 12^3 = 9^3 + 10^3 = 1729.
\end{equation}
Readers will recognize $1729$ as the famous Hardy-Ramanujan number, having the distinctive property that it is the smallest positive integer that can be written as the sum of two positive cubes in more than one way \cite{muk,schumer,silverman}. On the other hand, for $u=6$ and $v=-1$, we obtain
\begin{equation}\label{ex22}
10^3 + 783^3 + 953^3 = 1104^3.
\end{equation}

\item Likewise, using $(a,b,c,d) = (7,14,17,20)$ into Equations \eqref{sandor} and dividing by $6$ yields
\begin{align}\label{ex3}
(28u^2 +34uv -17v^2)^3 & + (56u^2 -40uv +20v^2)^3 \notag \\
& + (68u^2 -34uv -7v^2)^3 = (80u^2 -40uv + 14v^2)^3.
\end{align}
Putting $u = -2$ and $v = -3$ in \eqref{ex3} we find
\begin{equation}\label{ex31}
163^3 + 164^3 + 5^3 = 206^3.
\end{equation}
And, for $u=10$ and $v=3$, we obtain
\begin{equation}\label{ex32}
3667^3 + 4580^3 + 5717^3 = 6926^3.
\end{equation}
\end{enumerate}

Rather interestingly, it can be shown (see \cite[Theorem 2]{sandor}) that if $(a,b,c,d)$ is a nontrivial integer solution to \eqref{cubic} and $(q_1,q_2,q_3,q_4)$ is a nontrivial integer solution obtained via Equations \eqref{sandor}, then necessarily
\begin{equation}\label{fraction}
\frac{a+c}{d-b} = \frac{q_1 + q_3}{q_4 - q_2}.
\end{equation}
Note that the denominators in Equation \eqref{fraction} are well defined since both $(a,b,c,d)$ and $(q_1,q_2,q_3,q_4)$ are
nontrivial solutions. Note further that, for given $a$, $b$, $c$, and $d$, relation \eqref{fraction} holds irrespective of the (integer) values taken by $u$ and $v$ in Equations \eqref{sandor}. Conversely (see \cite[Theorem 3]{sandor}), if $(a,b,c,d)$ and $(q_1,q_2,q_3,q_4)$ are two integer nontrivial solutions to \eqref{cubic} and $\frac{a+c}{d-b} = \frac{q_1 + q_3}{q_4 - q_2}$, then there exist integers $u$ and $v$ such that substituting these into Equations \eqref{sandor} yields $(q_1,q_2,q_3,q_4)$ or a multiple of it. Thus, the solutions $(q_1,q_2,q_3,q_4)$ obtained from a given solution $(a,b,c,d)$ via Equations \eqref{sandor} can be essentially characterized through the condition stated in Equation \eqref{fraction}.

We can readily check that the said condition is indeed satisfied by the examples given above. So, for the example given in \eqref{ex2}, the condition \eqref{fraction} reads as
\begin{equation*}
\frac{1+8}{9-6} = \frac{3u^2 +15uv -8v^2 + 24u^2 -15uv -v^2}{27u^2 -21uv + 6v^2 - 18u^2 + 21uv -9v^2} = 3,
\end{equation*}
which is fulfilled for all integer values of $u$ and $v$ (discarding the trivial solution $u=v=0$). In particular, it holds for the numerical examples \eqref{ex21} and \eqref{ex22}, as $\frac{1-10}{9-12} = \frac{10+953}{1104-783} =3$. Similarly, regarding the example given in \eqref{ex3}, the condition \eqref{fraction} reads as
\begin{equation*}
\frac{7+17}{20-14} = \frac{28u^2 +34uv -17v^2 + 68u^2 -34uv -7v^2}{80u^2 -40uv + 14v^2 - 56u^2 +40uv -20v^2} = 4,
\end{equation*}
which is equally fulfilled for any choice of integers $u$ and $v$ (excluding the case $u=v=0$). In particular, it holds for the
numerical examples \eqref{ex31} and \eqref{ex32}, as $\frac{163+5}{206-164} = \frac{3667+5717}{6926-4580} =4$.

We conclude this section by noting that, naturally, given an initial solution $(a,b,c,d)$ to \eqref{cubic}, we can use as well either of its permutations $(a,c,b,d)$, $(b,a,c,d)$, $(b,c,a,d)$, $(c,a,b,d)$, or $(c,b,a,d)$ as inputs to Equations \eqref{sandor}. For instance, instead of the initial solution $(1,6,8,9)$ used above, we can plug $(1,8,6,9)$ into Equations \eqref{sandor} to obtain
\begin{align}\label{next}
(7u^2 +17uv -6v^2)^3 & + (56u^2 -35uv +9v^2)^3 \notag \\
& + (42u^2 -17uv -v^2)^3  = (63u^2 -35uv + 8v^2)^3.
\end{align}
Incidentally, we observe that setting $u=-1$ and $v=-3$ in \eqref{next} yields (after dividing by 2): $2^3 + 16^3 = 9^3 + 15^3 = 4104$, which is the next Hardy-Ramanujan number after $1729$. We encourage the students to search for their own solutions to the cubic equation \eqref{cubic} by means of Theorem \ref{th:2.1}, and to verify that they comply with relation \eqref{fraction}.

\section{Quadratic forms of sums of powers of integers}\label{sec:3}

Consider now the power sums $S_k = 1^k + 2^k + \cdots + n^k$ and $S_m = 1^m + 2^m + \cdots + n^m$. Their product is
given by
\begin{equation}\label{product}
S_k S_m = \frac{1}{k+1} \sum_{j=0}^{k/2} B_{2j} \binom{k+1}{2j} S_{k+m+1-2j} +
\frac{1}{m+1} \sum_{j=0}^{m/2} B_{2j} \binom{m+1}{2j} S_{k+m+1-2j},
\end{equation}
where $B_0 =1$, $B_1 = -1/2$, $B_2 = 1/6$, $B_3 =0$, $B_4 = -1/30$, etc., are the Bernoulli numbers (which fulfill the property
that $B_{2j+1} =0$ for all $j \geq 1$) \cite{apostol,deeba}; $\binom{k}{m}$ are the familiar binomial coefficients; and where the
upper summation limit $k/2$ denotes the greatest integer lesser than or equal to $k/2$. Formula \eqref{product} is not commonly
encountered across the abound literature on sums of powers of integers. A notable exception being the paper \cite{macdougall}, where formula \eqref{product} is stated as a theorem. As noted in \cite{macdougall}, formula \eqref{product} was known to Lucas by 1891. For the case that $k =m$, formula \eqref{product} reduces to
\begin{equation}\label{square}
S_k^2 = \frac{2}{k+1} \sum_{j=0}^{k/2} B_{2j} \binom{k+1}{2j} S_{2k+1-2j}.
\end{equation}
For later reference, we also quote the formula for the $k$-th power of $S_1$ expressed as a linear combination of power sums
\begin{equation}\label{s1k}
S_1^k = \frac{1}{2^{k-1}} \sum_{j=0}^{\frac{k-1}{2}} \binom{k}{2j+1} S_{2k-1-2j},
\end{equation}
as well as the formula for the product
\begin{equation}\label{s2s1k}
S_2 S_1^k = \frac{1}{3 \cdot 2^k} \sum_{j=0}^{\frac{k+1}{2}} \frac{2k+3-2j}{2j+1} \binom{k+1}{2j} S_{2k+2-2j}.
\end{equation}
Note that the right-hand side of Equations \eqref{square} and \eqref{s1k} involves only power sums $S_j$ with $j$ odd, whereas that of Equation \eqref{s2s1k} involves only power sums $S_j$ with $j$ even. Formula \eqref{s1k} (written in a slightly different form) appears as a theorem in \cite{macdougall}, where it is further noted that it was known as far back as 1877 (Lampe) and 1878 (Stern). Regarding formula \eqref{s2s1k}, it looks somewhat more exotic, although it is by no means new. An equivalent formulation of both Equations \eqref{s1k} and \eqref{s2s1k} can be found in, respectively, formulas (17) and (22) of the review paper by Kotiah \cite{kotiah}. It is worth pointing out, on the other hand, that the right-hand side of Equation \eqref{s1k} [\eqref{s2s1k}] can be interpreted as a sort of average of sums of powers of integers as the total number $\sum_{j=0}^{\frac{k-1}{2}} \binom{k}{2j+1}$ [$\sum_{j=0}^{\frac{k+1}{2}}\frac{2k+3-2j}{2j+1} \binom{k+1}{2j}$] of power sums appearing on the right-hand side of \eqref{s1k} [\eqref{s2s1k}] is just $2^{k-1}$ [$3\cdot 2^k$] (see \cite{cer,piza}).

Provided with Equations \eqref{product} and \eqref{square}, we can thus write the quadratic form \eqref{Qi} as the following linear combination of power sums:
\begin{align}\label{Qi2}
& Q_i(k,m,n)  =  \frac{2\alpha_i}{k+1} \sum_{j=0}^{k/2} B_{2j} \binom{k+1}{2j} S_{2k+1-2j} +
\frac{\beta_i}{k+1} \sum_{j=0}^{k/2} B_{2j} \binom{k+1}{2j} S_{k+m+1-2j} \notag  \\
& \qquad\quad + \frac{\beta_i}{m+1} \sum_{j=0}^{m/2} B_{2j} \binom{m+1}{2j} S_{k+m+1-2j} +
\frac{2\gamma_i}{m+1} \sum_{j=0}^{m/2} B_{2j} \binom{m+1}{2j} S_{2m+1-2j}.
\end{align}
Regarding the coefficients $\alpha_i$, $\beta_i$, and $\gamma_i$, we must choose them so that the quadratic forms $q_i = \alpha_i u^2 + \beta_i uv + \gamma_i v^2$ satisfy the relation $q_1^3 + q_2^3 + q_3^3 = q_4^3$ for any integer values of $u$ and $v$. This in turn ensures that the quadratic forms $Q_i(k,m,n)$ in Equation \eqref{Qi2} will satisfy $Q_1(k,m,n)^3 + Q_2(k,m,n)^3 + Q_3(k,m,n)^3 = Q_4(k,m,n)^3$ as well.

At this point, it is obviously most useful to run a computer algebra system such as {\it Mathematica} to quickly compute the quadratic form $Q_i(k,m,n)$ for concrete values of $\alpha_i$, $\beta_i$, $\gamma_i$, $k$, $m$, and $n$. As a simple but illustrative example, let us first take $k=1$ and $m=2$ to obtain
\begin{equation*}
Q_i(1,2,n) = \frac{1}{6} \big( \beta_i S_2 + \big(6\alpha_i + 2\gamma_i \big) S_3 + 5 \beta_i S_4 + 4 \gamma_i S_5 \big).
\end{equation*}
Then, choosing for example the coefficients $\alpha_i$, $\beta_i$, and $\gamma_i$ appearing in Equation \eqref{ex2} (namely, $\alpha_1 = 3$, $\beta_1 = 15$, $\gamma_1 = -8$, $\alpha_2 = 18$, $\beta_2 = -21$, $\gamma_2 = 9$, $\alpha_3 =24$, $\beta_3 = -15$, $\gamma_3 = -1$, $\alpha_4 =27$, $\beta_4 = -21$, and $\gamma_4 = 6$), we get (after removing the common factor $1/6$) the following relationship among the power sums $S_2$, $S_3$, $S_4$, and $S_5$:
\begin{multline}\label{relation1}
\big( 15S_2 +2S_3 +75S_4 -32S_5 \big)^3 + \big( -21S_2 +126 S_3 -105S_4 + 36S_5 \big)^3  \\
+ \big( -15S_2 +142S_3 -75S_4 -4S_5 \big)^3 = \big( -21S_2 +174S_3 -105S_4 +24S_5 \big)^3 .
\end{multline}
Equation \eqref{relation1} can in turn be written explicitly as a function of the variable $n$ by expressing each of the involved power sums in terms of $n$. It is a well-known result that $S_k$ can be expressed as a polynomial in $n$ of degree $k+1$ with zero constant term according to the formula (see, for instance, \cite{kotiah,tanton,williams,wu}):
\begin{equation}\label{bernoulli}
S_k = \frac{1}{k+1} \sum_{j=1}^{k+1} \binom{k+1}{j} (-1)^{k+1-j} B_{k+1-j} n^j,  \qquad k \geq 0.
\end{equation}
This formula, which was first established by Jacob Bernoulli in his masterpiece {\it Ars Conjectandi\/} (published posthumously in 1713 \cite{alexander}), provides an efficient way to compute the power sums $S_k$. Table \ref{tb:1} shows the polynomials for $S_2, S_3,\ldots,S_7$ as obtained from Bernoulli's formula \eqref{bernoulli}.\footnote{
Equation \eqref{bernoulli} is often referred to in the literature as Faulhaber's formula after the German engineer and mathematician Johann Faulhaber (1580-1635). In our view, however, it is more accurate to name Equation \eqref{bernoulli} as Bernoulli's formula or Bernoulli's identity.} Thus, substituting the power sums $S_2$, $S_3$, $S_4$, and $S_5$ in Equation \eqref{relation1} by the corresponding polynomial in Table \ref{tb:1}, and renaming the variable $n$ as a generic variable $u$, we get (after multiplying by an overall factor of $3$) the
algebraic identity
\begin{align}\label{relation2}
& \big(32u^2 +93u^3 +74u^4 -3u^5 -16u^6 \big)^3 + \big(54u^2 +63u^3 -18u^4 -9u^5 +18u^6 \big)^3  \notag \\
& + \big(85u^2 +123u^3 -11u^4 -51u^5 -2u^6 \big)^3 = \big(93u^2 +135u^3 +3u^4 -27u^5 +12u^6 \big)^3 ,
\end{align}
which is true for all real or complex values of $u$. For example, for $u=1$, relation \eqref{relation2} gives us (after dividing by 36): $5^3 + 3^3 + 4^3 = 6^3$.

\begin{table}[ttt]
\centering
\begin{tabular}{l}
% \hline
$S_2 = \frac{1}{3} n^3 + \frac{1}{2} n^2 + \frac{1}{6} n $ \\
$S_3 = \frac{1}{4} n^4 + \frac{1}{2} n^3 + \frac{1}{4} n^2 $  \\
$S_4 = \frac{1}{5} n^5 + \frac{1}{2} n^4 + \frac{1}{3} n^3 - \frac{1}{30}n $  \\
$S_5 = \frac{1}{6} n^6 + \frac{1}{2} n^5 + \frac{5}{12} n^4 - \frac{1}{12}n^2 $ \\
$S_6 = \frac{1}{7} n^7 + \frac{1}{2} n^6 + \frac{1}{2} n^5 - \frac{1}{6}n^3 + \frac{1}{42}n$ \\
$S_7 = \frac{1}{8} n^8 + \frac{1}{2} n^7 + \frac{7}{12} n^6 - \frac{7}{24}n^4 + \frac{1}{12}n^2 $ \\
 % \hline
\vspace{-2mm}
\end{tabular}
\caption{\small{The power sums $S_2, S_3,\ldots,S_7$ expressed as polynomials in $n$}.}\label{tb:1}
\end{table}

On the other hand, taking $u = S_2$ and $v = S_1^k$ in the quadratic form $q_i = \alpha_i u^2 + \beta_i uv + \gamma_i v^2$ yields
\begin{equation*}
F_i(k,n) = \alpha_i S_2^2 + \beta_i S_2 S_1^k + \gamma_i S_1^{2k}.
\end{equation*}
Utilizing Equations \eqref{s1k} and \eqref{s2s1k}, and noting that $S_2^2 = \frac{1}{3}S_3 + \frac{2}{3}S_5$, we can then write $F_i(k,n)$
as the linear combination of power sums
\begin{multline}\label{Fi}
F_i(k,n) = \frac{\alpha_i}{3} S_3 + \frac{2\alpha_i}{3} S_5 + \frac{\beta_i}{3 \cdot 2^k} \sum_{j=0}^{\frac{k+1}{2}} \frac{2k+3-2j}{2j+1}
\binom{k+1}{2j} S_{2k+2-2j} \\
+ \frac{\gamma_i}{2^{2k-1}} \sum_{j=0}^{\frac{2k-1}{2}} \binom{2k}{2j+1} S_{4k-1-2j}.
\end{multline}
As before, in order to derive relations of the type $F_1(k,n)^3 + F_2(k,n)^3 + F_3(k,n)^3 = F_4(k,n)^3$, we must choose the coefficients $\alpha_i$, $\beta_i$, and $\gamma_i$ such that the quadratic forms $q_i = \alpha_i u^2 + \beta_i uv + \gamma_i v^2$ satisfy $q_1^3 + q_2^3 + q_3^3 = q_4^3$ for any integer values of $u$ and $v$. As a concrete example, let us first take $k=2$ in Equation \eqref{Fi} to get
\begin{equation*}
F_i(2,n) = \frac{1}{12}\big( 4\alpha_i S_3 + 5\beta_i S_4 + (8\alpha_i + 6\gamma_i ) S_5 + 7 \beta_i S_6 + 6 \gamma_i S_7 \big).
\end{equation*}
Then, using the coefficients $\alpha_i$, $\beta_i$, and $\gamma_i$ appearing in Equation \eqref{next} (namely, $\alpha_1 = 7$, $\beta_1 = 17$, $\gamma_1 = -6$, $\alpha_2 = 56$, $\beta_2 = -35$, $\gamma_2 = 9$, $\alpha_3 =42$, $\beta_3 = -17$, $\gamma_3 = -1$, $\alpha_4 =63$, $\beta_4 =-35$, and $\gamma_4 = 8$), we obtain (omitting the common factor $1/12$) the following relationship among the power sums $S_3$, $S_4$, $S_5$, $S_6$, and $S_7$:
\begin{equation}\label{relation3}
\begin{split}
& \big( 28S_3 + 85S_4 + 20S_5 + 119S_6 - 36S_7 \big)^3  \\
& \qquad + \big( 224S_3 - 175S_4 + 502S_5 - 245S_6 + 54S_7 \big)^3  \\
& \qquad +  \big( 168S_3 - 85S_4 + 330S_5 - 119S_6 - 6S_7 \big)^3  \\
& \qquad\qquad =  \big( 252S_3 - 175S_4 + 552S_5 - 245S_6 + 48S_7 \big)^3 .
\end{split}
\end{equation}
Likewise, replacing each of the power sums in \eqref{relation3} by its respective polynomial in Table \ref{tb:1} yields (after multiplying by an overall factor of 12) the algebraic identity
\begin{equation}\label{relation4}
\begin{split}
& \big( 28u^2 +270u^3 +820u^4 +1038u^5 +502u^6 -12u^7 -54u^8 \big)^3 \\
&  \qquad + \big( 224u^2 + 1134u^3 + 1943u^4 +1122u^5 -88u^6 -96u^7 +81u^8 \big)^3  \\
&  \qquad + \big( 168u^2 + 906u^3 + 1665u^4 +1062u^5 -96u^6 -240u^7 -9u^8 \big)^3  \\
&  \qquad\qquad  = \big( 252u^2 + 1302u^3 + 2298u^4 +1422u^5 -30u^6 -132u^7 +72u^8 \big)^3 ,
\end{split}
\end{equation}
which has been written in terms of the generic (complex or real) variable $u$. Moreover, it is to be noted that each of the four summands in Equation \eqref{relation4} can be factorized as $u^2 (u+1)^2$ times a polynomial in $u$ of degree 4, so that the identity in \eqref{relation4} can be neatly simplified to
\begin{equation*}
\begin{split}
& \big( 28 + 214 u + 364 u^2 + 96 u^3 - 54 u^4 \big)^3  + \big( 224 + 686 u + 347 u^2 - 258 u^3 + 81 u^4 \big)^3    \\
& + \big( 168 + 570 u + 357 u^2 - 222 u^3 - 9 u^4 \big)^3 = \big( 252 + 798 u + 450 u^2 - 276 u^3 + 72 u^4 \big)^3.
\end{split}
\end{equation*}
In particular, for $u=0$, we obtain $28^3 +224^3 + 168^3 = 252^3$ or, after dividing each term by 28, $1^3 + 8^3 + 6^3 = 9^3$.

Trivially, for $u=0$, all four summands in either of relations \eqref{relation2} or \eqref{relation4} vanish. Less obvious is the fact that the same happens for $u =-1$. To see why, we need to extend the domain of definition of $S_k = 1^k +2^k + \cdots+ n^k$ to negative values of $n$. As explained in \cite{hersh}, this can be achieved simply by subtracting successively the $k$-th power of $0$, $-1$, $-2$, etc. In this way, it is not difficult to show (see Table 1 of \cite{hersh}) that, for all $k \geq 1$, the polynomial $S_k$ is symmetric about the point $-\frac{1}{2}$ (see also \cite[Theorem 10]{newsome} for a rigorous proof of this assertion). Thus, as $S_k = 0$ for $n=0$, this means that $S_k$ equally vanishes for $n =-1$.\footnote{
It is left as an exercise to the reader to show the following basic recurrence formula for the Bernoulli numbers
\begin{equation*}
B_k = -\frac{1}{k+1} \sum_{j=0}^{k-1} \binom{k+1}{j} B_j,    \quad \text{for all}\,\, k \geq 1,
\end{equation*}
employing Bernoulli's formula \eqref{bernoulli}, and using that $S_k(-1)=0$ for all $k\geq 1$.} (It is readily verified that the polynomials in Table \ref{tb:1} indeed satisfy $S_j(-1)= 0$ for each $j=2,3,\ldots,7$.\footnote{
That $S_k(-1)=0$ also follows directly from the well-known fact that $S_1 = \frac{1}{2} n(n+1)$ is a factor of $S_k$ for all $k \geq1$.}) As a consequence, the quadratic forms $Q_i(k,m,n)$ and $F_i(k,n)$ (defined in Equation \eqref{Qi2} and \eqref{Fi}, respectively) are zero for $n=-1$, regardless of the values that $\alpha_i$, $\beta_i$, $\gamma_i$, $k$, and $m$ may take (provided that $k,m \geq 1$).

Again, we encourage the students to construct their own algebraic identities like those in Equations \eqref{relation2} and \eqref{relation4} by making use of the quadratic forms \eqref{Qi2} and \eqref{Fi}, and Bernoulli's formula \eqref{bernoulli}.

\section{Concluding remarks}\label{sec:4}

In what follows, we briefly consider the Diophantine quadratic equation
\begin{equation}\label{quadratic}
a^2 + b^2 + c^2 = d^2 .
\end{equation}
Quadruples of positive integers $(a,b,c,d)$ such as $(2,3,6,7)$ satisfying \eqref{quadratic} are called Pythagorean {\it quadruples}, in analogy with the Pythagorean {\it triples} $(a,b,c)$ satisfying $a^2 + b^2 = c^2$. A full account of Equation \eqref{quadratic}, including its most general solution, can be found in, for instance, \cite{oliverio,spira}. A partial, quadratic-form solution to \eqref{quadratic} was given by Titus Piezas III in \cite{piezas}
\begin{equation}\label{piezas}
\big( a u^2 -2d uv + av^2 \big)^2 + \big( bu^2 - bv^2 \big)^2 + \big( cu^2 - cv^2 \big)^2 = \big( du^2 -2auv +
dv^2 \big)^2,
\end{equation}
where $(a,b,c,d)$ is a Pythagorean quadruple and $u$ and $v$ are integer variables.\footnote{
It is to be noted that, for the specific case in which $b^2 + c^2 $ happens to be a perfect square, say $e^2$, Equation \eqref{piezas} becomes
\begin{equation*}
\big( a u^2 -2d uv + av^2 \big)^2 + \big( e u^2 - e v^2 \big)^2  = \big( du^2 -2auv + dv^2 \big)^2,
\end{equation*}
which constitutes a two-parameter solution to the Pythagorean equation $r^2 + s^2 = t^2$. For example, for $(a,b,c,d) =
(8,9,12,17)$, where $9^2 + 12^2 = 15^2$, we have
\begin{equation*}
\big( 8 u^2 - 34 uv + 8v^2 \big)^2 + \big( 15 u^2 - 15 v^2 \big)^2  = \big( 17u^2 -16 uv + 17v^2 \big)^2 .
\end{equation*}}

On the other hand, setting $u = S_k$ in the algebraic identity\footnote{
Clearly, Equation \eqref{quadratic2} is of the form $q_1^2 + q_2^2 + q_3^2 = q_4^2$, with each $q_i$ being a quadratic form $q_i = \alpha_i u^2 + \beta_i uv + \gamma_i v^2$. For example,
taking $v =1$, $\alpha_1 = \gamma_1 =0$, and $\beta_1 =1$, we have $q_1 =u$.}
\begin{equation}\label{quadratic2}
u^2 + (1 +u )^2 + ( u + u^2)^2 = ( 1 + u + u^2 )^2,
\end{equation}
and utilizing the formula \eqref{square}, we obtain the following solution to Equation \eqref{quadratic} in terms of sums of powers of integers:
\begin{multline}\label{quadratic3}
\big( S_k \big)^2 + \big( 1+S_k \big)^2 + \Bigg( S_k + \frac{2}{k+1} \sum_{j=0}^{k/2} B_{2j} \binom{k+1}{2j}
S_{2k+1-2j} \Bigg)^2   \\
= \Bigg( 1 + S_k + \frac{2}{k+1} \sum_{j=0}^{k/2} B_{2j} \binom{k+1}{2j} S_{2k+1-2j} \Bigg)^2.
\end{multline}
For example, for $k=2$, from Equation \eqref{quadratic3} we find (after multiplying by an overall factor of $3$)
\begin{equation*}
\big(3S_2 \big)^2 +  \big(3+3S_2 \big)^2 + \big(3S_2 +S_3 +2S_5 \big)^2 = \big(3+3S_2 +S_3 +2S_5 \big)^2 .
\end{equation*}
Now, replacing $S_2$, $S_3$, and $S_5$ by its respective polynomial in Table \ref{tb:1}, and multiplying by an overall factor of $12$, we arrive at the following identity
\begin{equation*}
a^2 + ( a +18)^2 + b^2 = (b + 18)^2,
\end{equation*}
where
\begin{equation*}
a = 3u +9u^2 +6u^3,  \quad\text{and}\quad  b =  3u +\frac{19}{2}u^2 +9u^3 +\frac{13}{2}u^4 +6u^5 +2u^6,
\end{equation*}
with $u$ taking any real or complex value. (It is easily seen that $b$ is integer whenever $u$ so is.)

Let us finally mention that, by replacing $u$ with $S_k^2$ and $v$ with $S_m^2$ in the basic identity $(u -v)^2 + (2\sqrt{uv})^2
= (u+v)^2$, one can generate infinite Pythagorean triangles through the relation
\begin{equation}\label{triple}
\big( S_k^2 - S_m^2 \big)^2 + \big( 2 S_k S_m \big)^2 = \big( S_k^2 + S_m^2 \big)^2 .
\end{equation}
Using Equations \eqref{product} and \eqref{square}, the side lengths of the triangle can furthermore be written as a linear combination of power sums. For example, for $k=1$ and $m=3$, from Equation \eqref{triple} we get (after multiplying by a global factor of $2$) the relation
\begin{equation*}
\big( S_5 + S_7 - 2S_3  \big)^2 + \big( S_3 +3S_5 \big)^2 = \big( 2S_3 + S_5 + S_7 \big)^2.
\end{equation*}
This is to be compared with the following relation
\begin{equation*}
\big( 2S_5  + 2S_7 -S_3 \big)^2 + \big( S_3 +3S_5 \big)^2 = \big( S_3 + 2S_5 + 2S_7 \big)^2,
\end{equation*}
which was derived by Piza \cite{piza} using the algebraic identity $( y^4 - y^2)^2/4 + (2y^3 )^2/4 = (y^4 +y^2)^2/4$ and then
taking $y =2S_1$. Now, from the last two relations, we readily obtain
\begin{equation*}
\big( S_5 + S_7 -2S_3  \big)^2 + \big( S_3 +2S_5 +2S_7 \big)^2 = \big( 2S_3 +S_5 +S_7 \big)^2
+ \big( 2S_5 +2S_7 -S_3 \big)^2.
\end{equation*}
Taking into account that $S_5 + S_7 = 2S_3^2$, this relation can be simplified to (after dividing by the common factor $S_3$):
\begin{equation}\label{abcd}
\big( 2u -2 \big)^2 + \big( 4u +1 \big)^2 = \big( 2u +2 \big)^2 + \big( 4u -1 \big)^2,
\end{equation}
with $u = S_3$. The identity in Equation \eqref{abcd}, which actually holds for arbitrary values of $u$, gives us a family of solutions to the Diophantine equation $a^2 + b^2 = c^2 + d^2$. For example, for $u =17$, from Equation \eqref{abcd} we find that $32^2 + 69^2 = 36^2 + 67^2$.

\section*{Acknowledgment}

The author thanks the referee for carefully going through the manuscript and for valuable comments.

\printbibliography

\end{document}